\documentclass[11 pt]{amsart}
\usepackage[utf8]{inputenc}

\usepackage{amsmath}
\usepackage{amssymb}
\usepackage{enumitem}
\usepackage{amsthm}
\usepackage{color}
\usepackage{comment}
\usepackage{graphicx}
\usepackage{a4wide}

\usepackage{mathabx} 

\numberwithin{equation}{section}

\newtheorem{theorem}{Theorem}[section]
\newtheorem*{theoremintro}{Theorem}
\newtheorem{proposition}[theorem]{Proposition}
\newtheorem{lemma}[theorem]{Lemma}
\newtheorem{corollary}[theorem]{Corollary}
\theoremstyle{definition}
\newtheorem{definition}[theorem]{Definition}
\theoremstyle{remark}
\newtheorem*{remark}{Remark}

\newcommand{\C}{\mathbb{C}}

\newcommand{\QG}{\mathbb{G}}
\newcommand{\QH}{\mathbb{H}}

\newcommand{\Pol}{\textup{Pol}}

\newcommand{\R}{\mathbb{R}}
\newcommand{\Z}{\mathbb{Z}}
\newcommand{\N}{\mathbb{N}}

\begin{document}

\title{Gaussian generating functionals on easy quantum groups}

\author{Uwe Franz}
\address{U.F., D\'epartement de math\'ematiques de Besan\c{c}on,
Universit\'e de Franche-Comt\'e, 16, route de Gray, 25 030
Besan\c{c}on cedex, France}
\email{uwe.franz@univ-fcomte.fr}
\urladdr{http://lmb.univ-fcomte.fr/uwe-franz}

\author{Amaury Freslon}
\address{Universit\'e Paris-Saclay, CNRS, Laboratoire de Math\'ematiques d'Orsay, 91405 Orsay, France}
\email{amaury.freslon@universite-paris-saclay.fr}

\author{Adam Skalski}
\address{Institute of Mathematics of the Polish Academy of Sciences, ul.~\'Sniadeckich 8, 00--656 Warszawa, Poland}
\email{a.skalski@impan.pl}

\thanks{A.S.\ was partially supported by the  National Science Centre
(NCN) grant no. 2020/39/I/ST1/01566 and by a grant from the Simons Foundation. A.S.\ would like to thank the Isaac Newton Institute for Mathematical Sciences, Cambridge, for support and hospitality during the programme Quantum Information, Quantum Groups and Operator Algebras, where work on this paper was completed. This  was supported by EPSRC grant no. EP/R014604/1.
}

\begin{abstract}
We describe all Gaussian generating functionals on several easy quantum groups given by non-crossing partitions. This includes in particular the free unitary, orthogonal and symplectic quantum groups. We further characterize central Gaussian generating functionals and describe a centralization procedure yielding interesting (non-Gaussian) generating functionals.
\end{abstract}
\maketitle


\section{Introduction}
\label{sec:intro}

The study of stochastic processes on (locally) compact groups naturally focuses primarily on the most natural class, that  of \emph{L\'evy processes}, i.e.\ stationary processes with independent, identically distributed increments. These can be naturally described via the associated convolution semigroups of probability measures, and further via their stochastic generators, which can be fully classified via the L\'evy-Khintchin formula in its various incarnations (\cite{Heyer}, \cite{liao}).

The arrival of quantum probability, and parallel developments related to quantum group theory, led in the 1980s to the emergence of the theory of \emph{quantum} L\'evy processes \cite{ASW}. They can be studied from a purely algebraic point of view, as -- similarly to their classical counterparts -- all the stochastic information they carry is contained in the associated \emph{(quantum) convolution semigroup}, which this time is a family of states on the underlying $*$-bialgebra, or in its \emph{generating functional} (\cite{schurmann93}).

Among classical L\'evy processes the most important are the Gaussian processes. The desire to understand their quantum equivalents led Sch\"urmann to introduce the notion of a  \emph{Gaussian generating functional}. These can be in fact defined and studied on any \emph{augmented algebra} (i.e.\ a complex unital $*$-algebra $A$ equipped with a character $\varepsilon : A \to \C$). We will however mostly focus on the original context of the algebras associated with compact quantum groups in the sense of Woronowicz, denoted below $\Pol(\QG)$, and further simply speak of Gaussian generating functionals on $\QG$, where $\QG$ is a given compact quantum group. The study of Gaussian generating functionals has many interesting connections to general quantum stochastic considerations (\cite{dg12}), to L\'evy-Khintchin decompositions and  cohomology questions  (\cite{dfks}) or to determining \emph{Gaussian parts} of certain quantum groups (\cite{FFS23}). In this article we will focus on the classification problems for Gaussian functionals on certain \emph{free/universal} quantum groups (\cite{banicaspeicher2009}). The main result of our work can be summarised as the following statement.

\begin{theoremintro}
Every Gaussian generating functional $\phi$ on the free unitary quantum group $U_{N}^{+}$ admits a unique decomposition into the sum of a  ``drift'' part $D_{H}$ determined by an anti-hermitian matrix $H\in M_{N}(\C)$ and a ``diffusion'' part $\Gamma_{W}$ determined by a matrix $W \in M_{N}(\C)\otimes M_{N}(\C)$ satisfying a certain positivity condition. Conversely, every pair $H,W$ as above leads to a Gaussian generating functional $\phi = D_{H} +\Gamma_{W}$.
\end{theoremintro}

The result, which has obvious classical analogues, allows us further to obtain analogous descriptions for free orthogonal groups, for free symplectic group and for certain other classes of easy/partition quantum groups. One should note that the theorem above might be helpful in solving the outstanding problem left open in \cite{FFS23}: do Gaussian processes ``see all of $U_{N}^{+}$''? Or, formally speaking, is $U_{N}^{+}$ its own Gaussian part?  

A particularly interesting class of quantum convolution semigroups (or more generally quantum probability measures) on compact quantum groups is given by \emph{central states} (see for example \cite{cfk} or \cite{FSW}). Motivated by this fact, we describe all central Gaussian generating functionals on the quantum groups listed above. These turn out to be rather limited; for example the free orthogonal group does not admit any central Gaussian processes. We show however how to produce central stochastic generators which are not Gaussian, but from a certain perspective can be viewed as quantum analogues of classical Brownian motions. These have been very recently studied in \cite{Delhaye}, with natural cut-off estimates obtained.

The detailed plan of the paper is as follows: in Section \ref{sec-prelim} we recall the basic notions and background results we need to study Gaussian generating functionals. Section \ref{Sec:unitary} treats the free unitary group and establishes the key theorem mentioned above. There we also discuss the Gaussian generating functionals on the infinite quantum hyperoctahedral group and on the duals of classical free groups. In Section \ref{Sec:orthogonal} we consider free orthogonal group and its symplectic counterpart. Section \ref{sec:central} is devoted to the analysis of central Gaussian functionals. We first show that they appear relatively rarely, and then we describe the centralizing procedure and apply it to Gaussian generators.

\section{Preliminaries}\label{sec-prelim}

In this preliminary section, we recall some basic definitions and facts regarding Gaussian generating functionals. 

The basic object of our study is an augmented unital complex $*$-algebra, i.e.\ a pair  $(B, \varepsilon)$, where $B$ is a unital $*$-algebra and $\varepsilon:B \to \C$  a character (unital $*$-homomorphism). For simplicity we will always call such a pair an \emph{augmented algebra}. A lot of the motivation for our study comes from the examples of augmented algebras given by the  Hopf $*$-algebra $\Pol(\QG)$ of a compact quantum group $\QG$ equipped with the counit, but we will stick for the moment to the general setting. A \emph{generating functional} on $B$ is a linear functional $\phi : B\to \C$ with the following three properties:
\begin{enumerate}
\item 
$\phi(\mathbf{1}) = 0$ (normalization);
\item 
$\phi(b^{*}) = \overline{\phi(b)}$ for all $b\in B$ (hermitianity);
\item
$\phi(b^{*}b)\geqslant 0$ for all $b\in \ker(\varepsilon)$ (conditional positivity).
\end{enumerate}

We are interested in generating functionals, because it follows from \cite[Section 3.2]{schurmann93} that if $B$ happens to be a $*$-bialgebra, such functionals are in one-to-one correspondence with \emph{convolution semigroups of states} (see for example \cite{dfks} for more information on the topic).

We now turn to the notion of Gaussianity for a generating functional. This is expressed through specific ideals of an augmented $*$-algebra $B$ which we now define. Set $K_{1}(B) = \ker(\varepsilon)$ and
\begin{equation*}
K_{n}(B) = \mathrm{Span}\{b_{1}\cdots b_{n} \mid b_{1}, \cdots, b_{n}\in \ker(\varepsilon)\} = K_{1}(B)^{n}.
\end{equation*}

\begin{definition}
A generating functional $\phi: B\to \C$ on an augmented $*$-algebra $B$ is called \emph{Gaussian} (or \emph{quadratic}, \cite[Section 5.1]{schurmann93}), if $\phi_{\mid K_{3}(B)} = 0$. 
\end{definition}

Note that the defining property of Gaussian generating functionals translates into the following condition, valid for all $a, b, c \in B$:
\begin{equation*}
\phi(abc) = \phi(ab) \varepsilon(c) + \phi(ac) \varepsilon(b) + \phi(bc) \varepsilon(a) - \phi(a) \varepsilon(bc) - \phi(b) \varepsilon(ac) - \phi(c) \varepsilon(ab).
\end{equation*}
This gives an inductive algorithm to compute $\phi$, see \cite[Prop 2.7]{FFS23}. Note also that Gaussian generating functionals form a cone inside $B'$. We need two more algebraic definitions.

\begin{definition}
Let $(B, \varepsilon)$ be an augmented $*$-algebra	and let $V$ be a vector space. A linear map $\eta: B \to V$ is called a \emph{Gaussian cocycle} (or an \emph{$\varepsilon$-derivation}) if for all $a, b\in B$,
\begin{equation*}
\eta(ab) = \varepsilon(a)\eta(b) + \eta(a)\varepsilon(b).
\end{equation*}
\end{definition}

\begin{definition}
Given an augmented algebra $(B, \varepsilon)$ and a functional $\psi : B \to \C$ we define the \emph{coboundary} of $\psi$, namely a functional $\partial \psi : B \otimes B \to \C$, by the (linear extension of the) formula
\begin{equation}\label{def:partial}
\partial \psi(a \otimes b)= 	\psi(ab) - \varepsilon(a)\psi(b) - \psi(a)\varepsilon(b),
\end{equation}
for $a, b\in B$.
\end{definition}

A version of the GNS construction, starting from a (Gaussian) generating functional shows the following facts (contained in \cite[Subsection 2.3]{schurmann93} and \cite[Proposition 5.1.1]{schurmann93}).

\begin{theorem} \label{GST}
Let $(B, \varepsilon)$ be an augmented $*$-algebra and let $\phi : B \to \C$ be a linear functional. The following facts are equivalent:
\begin{itemize}
\item[(i)] $\phi$ is a Gaussian generating functional;
\item[(ii)] $\phi$ is hermitian and there exist a pre-Hilbert space $D$ and a Gaussian cocycle $\eta : B \to D$ such that for all $a, b\in B$,
\begin{equation} \label{eq-cobound}
\partial\phi(a^{*}\otimes b) = \langle\eta(a), \eta(b)\rangle.
\end{equation}
\end{itemize}		
\end{theorem}

We can moreover assume that $\eta(B) = D$ (the cocycle $\eta$ is then called surjective). Given two surjective cocycles $\eta : B \to D, \eta' : B \to D'$ such that \eqref{eq-cobound} holds for both $\eta$ and $\eta'$, we have a natural unitary equivalence between $\eta$ and $\eta'$. 

\begin{definition} \label{def:GaussPair}
A pair $(\phi, \eta)$, where $\phi: B \to \C$ is a Gaussian generating functional and $\eta:B \to D$ is a   Gaussian cocycle such that \eqref{eq-cobound} holds, is called a \emph{Gaussian pair}. Given a Gaussian cocycle $\eta:B \to D$ we will say that it admits a Gaussian generating functional if there exists a Gaussian generating functional $\phi: B \to \C$ such that $(\phi, \eta)$ form a Gaussian pair.
\end{definition}

\begin{remark}
A special case of Gaussian pairs is that given by these of the form $(\phi,0)$. These arise from \emph{drifts} $\phi$: hermitian linear functionals which are at the same time $\varepsilon$-derivations. Note that drifts are Gaussian functionals, and given a Gaussian pair $(\phi, \eta)$ and a functional $\psi\in B'$ the pair $(\psi, \eta)$ is Gaussian if and only if $\psi-\phi$ is a drift. 
\end{remark}

The next proposition will be very useful in Section \ref{Sec:orthogonal}.

\begin{proposition} \label{prop:quotient}
Assume that $(B, \varepsilon)$ and $(\tilde{B}, \tilde{\varepsilon})$ are augmented $*$-algebras and that $q : \tilde{B}\to B$ is a morphism of augmented algebras, i.e.\ a unital $*$-homomorphism such that $\varepsilon\circ q = \tilde{\varepsilon}$. Then any Gaussian generating functional $\phi: B \to \C$ induces a Gaussian generating functional $\tilde{\phi} = \phi\circ q$ on $\tilde{B}$. If moreover $q$ is surjective, then there is a natural bijective correspondence between Gaussian pairs on $B$ and those Gaussian pairs on $\tilde{B}$ which vanish on $\ker(q)$.
\end{proposition}

\begin{proof}
This follows directly from the fact that $q(K_{n}(\tilde{B}))\subset K_{n}(B)$ for all $n\in \N$.
\end{proof}

To check that a Gaussian generating functional vanishes on a certain ideal determined by a set of relations,  we will rely on the following result proved in \cite[Cor 2.8]{FFS23} in the language of Hopf $*$-algebras of compact quantum groups. Here we will state it in a formally broader context of augmented $*$-algebras, but the proof remains exactly the same. 

\begin{lemma}\label{lem-ideal}
Let $(B, \varepsilon)$ be an augmented $*$-algebra and let $\phi : B\to \C$ be a Gaussian generating functional. Assume that we have  two families $\mathcal{X} = \{a_{1}, \cdots, a_{n}\}\subset B$ and $\mathcal{Y} = \{b_{1}, \cdots, b_{m}\}\subset \ker(\varepsilon)$ such that
\begin{itemize}
\item[(1)] the family $\mathcal{X}$ generates $B$ as an algebra;
\item[(2)] $0 = \phi(b_{k}) = \phi(a_{j}b_{k}) = \phi(b_{k}a_{j})$ for all $j = 1, \cdots, n$, $k = 1, \cdots, m$.
\end{itemize}
Then $\phi$ vanishes on the ideal generated by $\mathcal{Y}$.
\end{lemma}

Our goal in the sequel is to completely classify Gaussian functionals on augmented algebras associated with concrete compact quantum groups. The specific quantum groups that we will study are called \emph{free easy quantum groups} and were introduced under that name in \cite{banicaspeicher2009}. However, we will not need their general definition, because it will be more convenient to work with a specific description in each case. Moreover, there are many of these for which the problem has already been solved. Indeed, it was proven in \cite[Prop 4.10]{FFS23} that if the Hopf $*$-algebra $\Pol(\QG)$ of a compact quantum group is generated by projections, then $\Pol(\QG)$ admits no non-zero Gaussian functionals. This is in particular the case for the quantum permutation groups $S_{N}^{+}$ (see \cite{wang98} for the definition) and the quantum reflection groups $H_{N}^{s+}$ for $1\leqslant s < \infty$ (see \cite{BV09} for the definition). Note also that we will often simply speak of Gaussian pairs or Gaussian functionals on $\QG$ (as opposed to on $\Pol(\QG)$), as mentioned already in the introduction.

\section{Free unitary quantum groups}
\label{Sec:unitary}

We start our study with free unitary quantum groups, which were originally defined in \cite{wang95}. We refer the reader to \cite{freslon23} for a detailed treatment of the theory and the definitions of the objects that we will use, but most of our computations only involve the defining relations of the corresponding $*$-algebras, which we now give. Throughout this section we fix $N\in \N$.

\begin{definition}
Let $\Pol(U_{N}^{+})$ be the universal $*$-algebra generated by $N^{2}$ elements $(u_{ij})_{1\leqslant i, j\leqslant N}$ such that the matrices $U = (u_{ij})_{i,j=1}^N$ and $\overline{U} = (u_{ij}^{*})_{i,j=1}^N$ are both unitary. It is easy to check that the formula $\varepsilon(u_{ij}) = \delta_{ij}$, for all $1\leqslant i, j \leqslant N$ determines a character on $\Pol(U_{N}^{+})$.
\end{definition}

To classify Gaussian pairs on $U_{N}^{+}$, we will first record an upgraded version of an observation contained already in \cite{dfks}.

\begin{lemma}\label{lem:deriv}
Any matrix $A = (a_{ij})_{1\leqslant i,j\leqslant N}\in M_{N}(\C)$ determines a (unique) $\varepsilon$-derivation $D_{A}: \Pol(U_{N}^{+})\to\C$ through the formula
\begin{equation*}
D_{A}(u_{ij}) = a_{ij} \text{ for all } 1\leqslant i,j\leqslant N.
\end{equation*}
Moreover we have
\begin{equation*}
D_{A}(u_{ij}^{*}) = - a_{ji} \text{ for all } 1\leqslant i,j\leqslant N.
\end{equation*}
Thus, $D_{A}$ is a drift if and only if $A = -A^{*}$.
\end{lemma}

\begin{proof}
The first statement is a special case of \cite[Prop 3.2]{dfks} (with $R = \mathrm{I}_{N})$. The second is a consequence of the formula $\eta\circ S = -\eta$, valid for any Gaussian cocycle on a Hopf $*$-algebra associated with a compact quantum group, with $S$ denoting the antipode (see the proof of \cite[Thm 3.11]{FFS23}), and the fact that $S(u_{ij}) = u_{ji}^{*}$ for all $1\leqslant i,j \leqslant N$. The last statement follows from the fact that the complex conjugate of a $\C$-valued $\varepsilon$-derivation is also an $\varepsilon$-derivation and the injectivity of the map $A \mapsto D_{A}$.
\end{proof}

The following is the first version of the main result of this section.

\begin{theorem}\label{thm-gauss-UNplus}
Let $d\in\mathbb{N}$, let $L_{1}, \cdots, L_{d}\in M_{N}(\C)$ be such that
\begin{equation}\label{eq-L-cond}
\sum_{r=1}^{d} L^{*}_{r}L_{r} = \sum_{r=1}^{d} L_{r}L^{*}_{r}
\end{equation}
and let $H\in M_{N}(\C)$ be anti-hermitian ($H = -H^{*}$). Denote by $(e_{1}, \cdots, e_{d})$ the usual orthonormal basis of $\C^{d}$. Then there exists a Gaussian cocycle $\eta : \Pol(U_{N}^{+}) \to \C^{d}$ and a hermitian functional $\Gamma : \Pol(U_{N}^{+}) \to \C$ such that for every $1\leqslant i,j \leqslant N$ we have
\begin{equation*}
\eta(u_{ij}) = \sum_{r=1}^{d} (L_{r})_{ij} e_{r} \quad \& \quad \Gamma(u_{ij}) = - \frac{1}{2}\sum_{r=1}^{d} (L_{r}^{*}L_{r})_{ij},
\end{equation*}
and for all $a, b\in \Pol(U_{N}^{+})$,
\begin{equation*}
\partial{\Gamma} (a^{*}\otimes b) = \langle \eta(a), \eta(b)\rangle.
\end{equation*}
The conditions above determine the pair $(\Gamma, \eta)$ uniquely. Set $\phi = \Gamma + D_{H}$. Then $(\phi, \eta)$ is a Gaussian pair, and moreover all Gaussian pairs with surjective cocycles on $\Pol(U_{N}^{+})$ (hence also all Gaussian generating functionals) arise in this way. We may in addition choose the matrices $L_{1}, \cdots, L_{d}$ to be linearly independent. 
\end{theorem}

\begin{proof}
Assume first that we are given matrices $L_{1}, \cdots, L_{d} \in M_{N}(\C)$ satisfying Equation \eqref{eq-L-cond} and $H\in M_{N}(\C)$ such that $H = -H^{*}$. Lemma \ref{lem:deriv} guarantees that for every $1\leqslant r\leqslant d$, we have an $\varepsilon$-derivation $D_{L_{r}}: \Pol(U_{N}^{+})\to \C$; it is then immediate that 
\begin{equation*}
\eta := \sum_{r=1}^{d} D_{L_{r}}e_{r} : \Pol(U_{N}^{+}) \to \C^{d} 
\end{equation*}
is also an $\varepsilon$-derivation.

Let us introduce matrices $B, \tilde{B} \in M_{N}(\C)$ through the formul\ae{} (for $1\leqslant i,j\leqslant N$)
\begin{equation}\label{eqBdef1}
B_{ij} = \sum_{p=1}^{N} \langle \eta(u_{ip}), \eta(u_{jp})\rangle = \sum_{r=1}^{d} \sum_{p=1}^{N} \overline{(L_{r})_{ip}} (L_{r})_{jp} = \sum_{r=1}^{d} (L_{r}L_{r}^{*})_{ji},
\end{equation}
\begin{equation}\label{eqBdef2}
\tilde{B}_{ij} = \sum_{p=1}^{N} \langle \eta(u_{ip}^{*}), \eta(u_{jp}^{*})\rangle = \sum_{r=1}^{d} \sum_{p=1}^{N} \overline{(L_{r})_{pi}} (L_{r})_{pj}  = \sum_{r=1}^{d}  (L_{r}^{*} L_{r})_{ij},
\end{equation}
where in the second string of equalities we use Lemma \ref{lem:deriv}.

The forward direction of \cite[Theorem 3.3]{dfks} states that if $B = \tilde{B}^{t}$ (which as the above computation shows is equivalent to Equation \eqref{eq-L-cond}), then $\eta$ admits a generating functional. Moreover the generating functional $\Gamma$ constructed in the proof of \cite[Theorem 3.3]{dfks} satisfies for all $1\leqslant i, j\leqslant N$ the equalities   
\begin{equation}\label{eqGammadef}
\Gamma(u_{ij}) = -\frac{1}{2} \tilde{B}_{ij} = -\frac{1}{2} \sum_{r=1}^{d} (L_{r}^{*} L_{r})_{ij},
\end{equation}
which is the condition displayed in the statement. The uniqueness claim follows as $\Gamma$ is assumed to be hermitian, so that the algebraic conditions determine both $\eta$ and $\Gamma$ in terms of their values on generators $u_{ij}$. Eventually, the last part of Lemma \ref{lem:deriv} shows that $D_{H}$ is a drift. Thus $(\phi, \eta)$ is a Gaussian pair.

Assume conversely that $\phi$ is a Gaussian functional on $\Pol(U_{N}^{+})$. By Theorem \ref{GST}, we can assume that it is a part of a Gaussian pair $(\phi, \eta)$, where $\eta: \Pol(U_{N}^{+}) \to D$ is surjective. As the derivation property implies that the image of $\eta$ is spanned by the images of the generators $u_{ij}$ for $1\leqslant i,j\leqslant N$, the space $D$ must be finite dimensional. Set $d = \dim(D)$ and identify $D$ with $\C^{d}$. Then, set for $1\leqslant r\leqslant d$, $1\leqslant i,j \leqslant N$
\begin{equation*}
(L_{r})_{i,j} = \langle e_{r}, \eta(u_{ij}) \rangle,
\end{equation*} 
so that
\begin{equation*}
\eta = \sum_{r=1}^{d} D_{L_{r}}e_{r}.
\end{equation*}
The argument in the first part of the proof shows that $\eta$ admits a generating functional $\Gamma : \Pol(U_{N}^{+})\to \C$ satisfying the conditions listed in the theorem. As $(\Gamma, \eta)$ is a Gaussian pair, by the remarks after Definition \ref{def:GaussPair} we must have $\Gamma = \phi + \omega$, where $\omega : \Pol(U_{N}^{+}) \to \C$ is a drift. Thus the last part of Lemma \ref{lem:deriv} ends the proof of the main statement of the theorem.

Eventually, let us prove that the cocycle
\begin{equation*}
\eta = \sum_{r=1}^{d} D_{L_{r}} e_r
\end{equation*}
is surjective if and only if the matrices  $L_{1}, \cdots, L_{d}$ are linearly independent. Indeed, as mentioned above, the range of $\eta$ is spanned by the elements $\eta(u_{ij})$, and it is easy to check that given a vector $\xi = (\lambda_{1}, \cdots, \lambda_{d})\in \C^{d}$, we have
\begin{equation*}
\big(\langle \xi, \eta(u_{ij}) \rangle\big)_{i,j=1}^{N} = \sum_{r=1}^{d} \overline{\lambda_{r}} L_{r},
\end{equation*}
hence $\xi \perp \mathrm{Ran}(\eta)$ if and only if $\sum \overline{\lambda_{r}} L_{r} = 0$, and the result follows.
\end{proof}

Note that given a fixed anti-hermitian $H\in M_{N}(\C)$, several tuples of linearly independent matrices $(L_{1}, \cdots, L_{d})$ as above can yield the same generating functional (and equivalent Gaussian pairs). The proof above shows that the freedom is related to the choice of the orthonormal basis in the carrier space of the cocycle $\eta$, which affects the matrices $L_{r}$ via a unitary transformation. In other words, given an alternative tuple of linearly independent matrices $(\tilde{L}_{1}, \cdots, \tilde{L}_{d})$ leading -- together with $H$ -- to the same $\phi$, we have a unitary $U \in M_{d}(\C)$ such that for all $1\leqslant i\leqslant d$, $\tilde{L}_{i} = \sum_{r=1}^{d} U_{ri} L_{r}$. Note that this operation does not affect the matrix
\begin{equation*}
W = \sum_{r=1}^{d} L_{r} \otimes L_{r}^{*} \in M_{N}(\C)\otimes M_{N}(\C).
\end{equation*}
This observation leads to the second main theorem of this section. We begin with a lemma characterising the matrices of the form mentioned above.

\begin{lemma}\label{lem:Choi}
Consider a matrix $W \in M_{N}(\C) \otimes M_{N}(\C)$,
\begin{equation*}
W = \sum_{i,j,k,l=1}^{n} W_{ij,kl} e_{ij} \otimes e_{kl}
\end{equation*}
for certain coefficients $W_{ij,kl}\in \C$. Then, the following conditions are equivalent:
\begin{itemize}
\item[(i)] there exists $d\in \N$ and matrices $L_{1}, \cdots, L_{d} \in M_{N}(\C)$ such that
\begin{equation*}
W = \sum_{r=1}^{d} L_{r}\otimes L_{r}^{*};
\end{equation*}
\item[(ii)] for every matrix $X\in M_{N}(\C)$, we have
\begin{equation*}
\sum_{i, j, k, l=1}^{N} \overline{X_{ik}} W_{ki,jl} X_{jl} \geqslant 0.
\end{equation*}
\end{itemize}
\end{lemma}

\begin{proof}
Let $A, B \in M_{N}(\C)$ and define a map $\Psi_{A \otimes B}: M_{N}(\C)\to M_{N}(\C)$ by the formula 
\begin{equation*}
\Psi_{A \otimes B}(Z) = AZB \text{ for } Z \in M_{N}(\C).
\end{equation*}
This extends by bilinearity to a linear map $\Psi_{W} : M_{N}(\C)\otimes M_{N}(\C)\to B(M_{N}(\C))$ which is easily seen to be injective. Applying the Kraus characterisation of completely positive maps on matrices, we see that $\Psi_{W}$ is completely positive if and only if $W$ satisfies (i). On the other hand, by Choi's theorem, $\Psi_{W}$ is completely positive if and only if its $N$-th matrix lifting maps the Choi matrix $E:=(e_{ij})_{i,j=1}^{N} \in M_{N}(M_{N}(\C))$ to a positive matrix. We therefore compute:
\begin{equation*}
\Psi_W^{(N)}(E) = (\Psi_{W}(e_{ij}))_{i, j=1}^{N} = \left(\sum_{k, l=1}^{N} W_{ki,jl} e_{kl}\right)_{i,j=1}^{N},
\end{equation*}
so that
$\Psi_{W}^{(N)}(E) \geqslant 0$ if and only if for any vectors $\xi_{1}, \cdots, \xi_{N} \in \C^{N}$ we have
\begin{equation*}
\sum_{i, j=1}^{N} \left\langle \xi_{i}, \sum_{k,l=1}^{N}W_{ki,jl} e_{kl} \xi_{j} \right\rangle \geqslant 0,
\end{equation*}
which is precisely condition (ii) in the lemma (set $X_{ij}= \langle e_i, \xi_j \rangle$ for $1\leqslant i, j\leqslant N$).
\end{proof}

\begin{remark}
Condition \eqref{eq-L-cond} can be expressed in terms of
\begin{equation*}
W = \sum_{r=1}^{d} L_{r}\otimes L_{r}^{*}
\end{equation*}
simply as $M(W) = M(\Sigma(W))$, where $M$ denotes the multiplication of two elements of the tensor product, and $\Sigma$ is the tensor flip. We can also express it using the map $\Psi_{W}$. Indeed, viewing $M_{N}(\C)$ as a Hilbert space, when equipped with the scalar product $\langle X, Y \rangle := \mathrm{Tr}(X^*Y)$ (for $X,Y \in M_{N}(\C)$), given a map $\Psi : M_{N}(\C) \to M_{N}(\C)$ we can also consider its Hilbert space adjoint $\Psi^{*} : M_{N}(\C) \to M_{N}(\C)$. Assume then that $W \in M_{N}(\C)\otimes M_{N}(\C)$ decomposes as $W = \sum_{i=1}^{r} L_{i}\otimes L_{i}^{*}$. Then, for all $X,Y \in M_{N}(\C)$,
\begin{equation*}
\langle X, \Psi_{W}(Y) \rangle = \mathrm{Tr}\left(X^{*}\sum_{i=1}^{d} L_{i}YL_{i}^{*}\right) = 
\sum_{i=1}^{r} \mathrm{Tr}(L_{i}^{*}X^{*}L_{i}Y)
\end{equation*}
so that $(\Psi_{W})^{*} = \Psi_{\Sigma(W)}$. Thus, $M(W) = M(\Sigma(W))$ if and only if $\Psi_{W}(1) = \Psi_{W}^{*}(1)$.
\end{remark}

We can now give the second form of our classification of Gaussian generating functionals on free unitary quantum groups.

\begin{theorem} \label{th:unitaryW}
There is a one-to-one correspondence between Gaussian generating functionals on $\Pol(U_{N}^{+})$ and pairs $(W, H)$ where
\begin{enumerate}
\item $H\in M_{N}(\C)$ is anti-hermitian;
\item $W = (w_{ijkl})_{1\leqslant i, j, k, l\leqslant N} \in M_{N}(\C)\otimes M_{N}(\C)$ satisfies $M(W) = M(\Sigma(W))$ and the positivity condition:
\begin{equation*}
\sum_{i, j, k, l=1}^{N} \overline{X_{ik}} W_{ki,jl} X_{jl} \geqslant 0.
\end{equation*}
for all $X = (X_{jk})_{1\leqslant j,k\leqslant N} \in M_{N}(\C)$.
\end{enumerate}
\end{theorem}

\begin{proof}
Given the pair $(W, H)$ as above we can use Lemma \ref{lem:Choi} to obtain matrices $L_{1}, \cdots, L_{d} \in M_{N}(\C)$ such that 
\begin{equation*}
W = \sum_{r=1}^{d} L_{r}\otimes L_{r}^{*};
\end{equation*}
as $M(W) = M(\Sigma(W))$, Condition \eqref{eq-L-cond} holds and we can apply Theorem \ref{thm-gauss-UNplus} to obtain a Gaussian generating functional $\phi = \Gamma + D_{H}$.

Conversely, given a Gaussian generating functional $\phi$, we obtain by Theorem  \ref{thm-gauss-UNplus} an anti-hermitian matrix $H$ and matrices $L_{1}, \cdots, L_{d}\in M_{N}(\C)$ such that $\phi = \Gamma + D_{H}$. We can simply set $W = \sum_{r=1}^{d} L_{r}\otimes L_{r}^{*}$.

It remains to show that the correspondence described above is  bijective; in other words, that given a Gaussian generating functional $\phi$ we can determine $(W, H)$. This could be deduced from the proof of Theorem \ref{thm-gauss-UNplus} and remarks after the theorem, but we can also argue directly. Assume that $\phi = \Gamma + D_{H}$ is constructed as in the first part of the proof. Then, $\Gamma$ is hermitian and by the formul\ae{} \eqref{eqBdef1}-\eqref{eqGammadef}, we have for each $1\leqslant i, j\leqslant N$
\begin{equation}\label{Gammaad}
\Gamma(u_{ij}) = - \frac{1}{2} \sum_{r=1}^{d} \sum_{p=1}^{N} \overline{(L_{r})_{pi}} (L_{r})_{pj} = \overline{- \frac{1}{2} \sum_{r=1}^{d} \sum_{p=1}^{N} \overline{(L_{r})_{pj}} (L_{r})_{pi}} = \overline{\Gamma(u_{ji})} = \Gamma(u_{ji}^{*}).
\end{equation}	
On the other hand 
\begin{equation*}
D_{H}(u_{ij}) - D_{H}(u_{ji}^{*}) = H_{ij} - \overline{D_{H}(u_{ji})} = H_{ij} - \overline{H_{ji}} = H_{ij} - (H^{*})_{ij} = 2H_{ij}.
\end{equation*}
Thus for every $1\leqslant i, j\leqslant N$ we have
\begin{equation}\label{phitoH}
H_{ij} = \frac{1}{2}(\phi(u_{ij}) - \phi(u_{ji}^{*})).
\end{equation}
Let then $\eta : \Pol(U_{N}^{+}) \to \C^{d}$ be a surjective cocycle such that $(\phi, \eta)$ is a Gaussian pair, and recall that we can write
\begin{equation*}
\eta = \sum_{r=1}^{d} D_{L_{r}}e_{r}.
\end{equation*}
Fix $i,j,k,l\in \{1, \cdots, N\}$ and consider the following computation:
\begin{align*}
\partial \phi(u_{ij} \otimes u_{kl}) & = \langle  \eta(u_{ij}^{*}), \eta(u_{kl} )\rangle = - \langle  \eta(u_{ji}), \eta(u_{kl} )\rangle = -\sum_{r=1}^{d} \overline{(L_{r})_{ji}} (L_{r})_{kl} \\
& = - \left(\sum_{r=1}^{d} L_{r}\otimes L_{r}^{*}\right)_{(k,i), (l,j)},
\end{align*}
so that in the end,
\begin{equation} \label{phitoW}
\partial \phi(u_{ij} \otimes u_{kl})= - W_{ki,lj}.
\end{equation}
\end{proof}

\begin{definition}
Given matrices $W \in M_{N}(\C)\otimes M_{N}(\C)$ and $H \in M_{N}(\C)$ satisfying the conditions stated in Theorem \ref{th:unitaryW} we will denote the Gaussian generating functional associated to $W$ and $0$ by $\Gamma_{W}$, so that any Gaussian generating functional $\phi$ on $\Pol(U_{N}^{+})$ decomposes uniquely as
\begin{equation*}
\phi = \Gamma_{W} + D_{H}.
\end{equation*}
We will then call $D_{H}$ the \emph{drift part} of $\phi$ and $\Gamma_{W}$ the \emph{diffusion part} of $\phi$. 
\end{definition}	

\begin{remark}
The space of anti-hermitian $N$ by $N$ complex matrices is nothing but the Lie algebra $\mathfrak{u}_{N}$ of the classical unitary group $U_{N}$. It is therefore in one-to-one correspondance with drifts on $U_{N}^{+}$. Moreover, one easily checks that if $\ast$ denotes the convolution of linear forms (induced by the coproduct), then
\begin{equation*}
D_{H}\ast D_{K} - D_{K}\ast D_{H} = D_{[H, K]}
\end{equation*}
so that drifts form a Lie algebra isomorphic to $\mathfrak{u}_{N}$.
\end{remark}

\begin{remark}
Note that the formul\ae{} \eqref{phitoH}-\eqref{phitoW} can be expressed via matrix liftings of $\phi$: if we write $U=(u_{ij})_{1\leqslant i,j\leqslant N} \in M_{N} (\Pol(U_{N}^{+}))$, and define $U\tilde{\otimes} U \in M_{N} (\C)\otimes M_{N}(\C) \otimes \Pol(U_{N}^{+})$ to be the matrix with entries $(U\tilde{\otimes} U)_{(i,k),(j,l)} = u_{ij}u _{kl}$, then we have 
\begin{align*}
H & = \frac{1}{2}\left(\phi^{(N)}(U) - \phi^{(N)}(U^{*})\right);\\
-W^{t \otimes t} & =  \phi^{(N^2)}(U \tilde{\otimes} U) - \phi^{(N)}(U) \otimes I_{N} - I_{N}\otimes \phi^{(N)} (U);\\
\phi^{(N)}(U) & = - \frac{1}{2} M(W) + H.
\end{align*}
\end{remark}

We will now briefly discuss a possible characterisation of the decomposition $\phi=\Gamma_W+D_H$. We know which generating functionals are drifts, but there is no general notion of a `driftless' (Gaussian) generating functional. But we have the following definition, writing for simplicity $K_i:=K_i(\Pol(U_N^+))$ with $i=1,2$.

\begin{definition}\label{def:drift}\cite[Remark 2.4]{fkls} Suppose that $V\subset \Pol(U_N^+)$ is a $*$-invariant vector  subspace such that $K_1=V\oplus K_2$. Denote by $P:K_1\to K_2$ the projection with respect to this decomposition. Then we say that a generating functional $\phi:\Pol(U_N^+) \to \C$ is \emph{driftless with respect to $P$} if it satisfies $\phi\circ P=\phi$.
\end{definition}

One can show that for any generating functional $\phi$ (on any $\Pol(\QG)$) there exists a projection with respect to which $\phi$ is driftless -- see \cite[Section 2.2: $K_1$ and $K_2$]{skeide}

\begin{lemma}\label{lem:k1k2}
The $N^2$ elements $u_{jk}-u^*_{kj}\in \Pol(U_N^+)$, $1\le j,k\le N$, are linearly independent.
Furthermore, if we denote $V=\mathrm{span}\{u_{jk}-u^*_{kj}; 1\le j,k\le N\}$, then $K_1=V\oplus K_2$ is a direct sum of $*$-invariant vector spaces.
\end{lemma}

\begin{proof}
Set $\hat{u}_{jk} = u_{jk}-\delta_{jk}\mathbf{1}$ for $j,k=1,\ldots,N$. It is easy to see that these $N^2$ elements generate the ideal $K_1$. Unitarity of $U$ implies that
\[
- \sum_{\ell=1}^N \hat{u}_{j\ell}\hat{u}^*_{k\ell} = \hat{u}_{jk}+\hat{u}^*_{kj} = - \sum_{\ell=1}^N \hat{u}^*_{\ell j}\hat{u}_{\ell k}, \qquad j,k=1,\ldots,N,
\]
which shows that the elements $\hat{u}_{jk}+\hat{u}^*_{kj}$, $j,k=1,\ldots,N$ belong to $K_2$. Since $\hat{u}_{jk}-\hat{u}^*_{kj}=u_{jk}-u^*_{kj}$, it follows that 
$V+K_2=K_1$.

Suppose we have coefficients $\lambda_{jk}\in\mathbb{C}$ ($1\leqslant, j,k \leqslant N$) and $a\in K_2$ such that
\[
\sum_{j,k=1}^N \lambda_{jk} \left(u_{jk}-u^*_{kj}\right) + a = 0.
\]
  Applying the $\varepsilon$-derivation $D_{e_{st}}$, $1\le s,t\le N$, and using Lemma \ref{lem:deriv}, we find
\[
0 = D_{e_{st}}\left(\sum_{j,k=1}^N \lambda \left(u_{jk}-u^*_{kj}\right) + a \right) = -2 \lambda_{st},
\] 
since any $\varepsilon$-derivation vanishes on $K_2$. This  implies that $K_2 \cap V =\{0\}$ and proves that the elements $u_{jk}-u^*_{kj}$, $j,k=1,\ldots,N$, are linearly independent.

Finally, since the differences $u_{jk}-u^*_{kj}$ are anti-hermitian and since $K_2$ is a $*$-ideal, it is also clear that both $V$ and $K_1$ are invariant under the involution.
\end{proof}

\begin{proposition}
Denote by $P:K_1\to K_2$ the projection defined by the decomposition in Lemma \ref{lem:k1k2}. Let $\phi$ be a Gaussian generating functional on  $\Pol(U_{N}^{+})$. Then the decomposition
$\phi = \Gamma_W + D_H$
from Theorem \ref{th:unitaryW} is the unique decomposition of $\phi$ into a drift and a Gaussian generating functional which is driftless with respect to $P$.
\end{proposition}
\begin{proof}
Equation \eqref{Gammaad} shows that $\Gamma$ is driftless with respect to $P$, i.e.\ that it vanishes on $V$. And it is clearly the only such generating functional that agrees with $\phi$ on $K_2$.
\end{proof}

Before going further, let us connect our classification of Gaussian processes on $U_{N}^{+}$ to the classical unitary group $U_{N}$. Recall that the usual unitary group $U_{N}$ can be viewed as a (quantum) subgroup of $U_{N}^{+}$ via the map $q:\Pol(U_{N}^{+}) \to \Pol(U_{N})$ given by quotienting out the commutator ideal.

\begin{proposition}\label{prop:classical}
Assume that $\phi = \Gamma_{W} + D_{H}: \Pol(U_{N}^{+})\to \C$ is a Gaussian generating functional for $W\in M_{N}(\C)\otimes M_{N}(\C), H\in M_{N}(\C)$ such as in Theorem \ref{th:unitaryW}. Then, $\phi$ factors through $U_{N}$ if and only if $W = \Sigma(W)$.
\end{proposition}

\begin{proof}
If $\phi$ vanishes on the commutator ideal, then we must have $\phi(u_{ij}u_{kl}) - \phi(u_{kl}u_{ij}) = 0$ for all $1\leqslant i,j,k,l\leqslant N$. Using \eqref{def:partial} and \eqref{eq-cobound}, we see that the last condition is equivalent to
\begin{equation*}
\partial \phi(u_{ij} \otimes u_{kl}) = \partial \phi(u_{kl} \otimes u_{ij}),
\end{equation*}
which by \eqref{phitoW} is further equivalent to $W = \Sigma(W)$.

Assume then that $W = \Sigma(W)$ and $H$ are as in Theorem \ref{th:unitaryW}. We want to prove that $\phi$ factors via $U_{N}$ (i.e.\ it vanishes on the commutator ideal of $\Pol(U_{N}^{+})$). We will first show that for any $1\leqslant i,j,k,l,m,n \leqslant N$, we have
\begin{equation}\label{traceideal1}
\phi\left(u_{mn}(u_{ij}u_{kl} - u_{kl}u_{ij})\right) = 0  = \phi\left(u_{mn}^{*} (u_{ij}u_{kl} - u_{kl}u_{ij})\right).
\end{equation} 
Again, as by the argument in the first part of the proof we have $\phi(u_{ij}u_{kl}) = \phi(u_{kl}u_{ij})$, the displayed formul\ae{} are equivalent to 
\[ \partial \phi(u_{mn} \otimes (u_{ij}u_{kl} - u_{kl}u_{ij})) =0  = \partial \phi(u_{mn}^* \otimes (u_{ij}u_{kl} - u_{kl}u_{ij})). \] 
But now, using \eqref{eq-cobound} and noting that every $\varepsilon$-derivation is tracial we see that the above holds. Analogous arguments show that 
\begin{align}\label{traceideal2}
\phi\left(u_{mn}(u_{ij}^{*}u_{kl} - u_{kl}u_{ij}^{*})\right) & = 0  = \phi\left(u_{mn}^{*}(u_{ij}^{*}u_{kl} - u_{kl}u_{ij}^{*})\right) \\ \label{traceideal3}
\phi\left((u_{ij}^{*}u_{kl} - u_{kl}u_{ij}^{*}) u_{mn}\right) & = 0  = \phi\left((u_{ij}^{*}u_{kl} - u_{kl}u_{ij}^{*}) u_{mn}\right) \\ \label{traceideal4}
\phi\left((u_{ij}^{*}u_{kl} - u_{kl}u_{ij}^{*}) u_{mn}^{*}\right) & = 0  = \phi\left((u_{ij}^{*}u_{kl} - u_{kl}u_{ij}^{*}) u_{mn}\right)
\end{align} 
In view of the equalities \eqref{traceideal1}-\eqref{traceideal4}, Lemma \ref{lem-ideal} applied to $\mathcal{X}=\{u_{ij}, u_{ij}^{*} \mid 1\leqslant i,j \leqslant N\}$ and $\mathcal{Y} = 
\{u_{ij} u_{kl} - u_{kl} u_{ij},u_{ij} u_{kl}^{*} - u_{kl}^{*} u_{ij} \mid 1\leqslant i,j,k,l\leqslant N\}$ ends the proof.
\end{proof}

\begin{remark}
The family of all free unitary quantum groups includes many examples which are not of so-called \emph{Kac type}. However, it was proved in \cite{FFS23} that Gaussian pairs on a compact quantum group always factor through the maximal Kac type quantum subgroup. The relevant `Kac quotients' for free unitary quantum groups  were computed in \cite{soltan05}, and shown to involve free products of copies of $U_{N}^{+}$ for various values of $N$. However we do not have a general formula for Gaussian pairs on free products of augmented algebras, and computing all Gaussian generating functionals on such free products seems difficult with the tools available at the time of this writing.
\end{remark}

Before turning to the orthogonal case, we will use our results to classify Gaussian pairs on another family of free easy quantum groups, namely the infinite quantum hyperoctahedral groups $H_{N}^{\infty +}$. This was introduced in \cite{BV09} but all that we need to know is that, as proven in \cite[Prop 5.6]{FFS23}, its Gaussian part is the same as the Gaussian part of the free group on $N$ generators $\mathbb{F}_{N}$. This means that any Gaussian pair on $\Pol(H_{N}^{\infty +})$ factors through $\C[\mathbb{F}_{N}]$ so that we only have to describe Gaussian pairs on free groups. It turns out that this is easy using Theorem \ref{thm-gauss-UNplus}, but we first need to properly define the connection between $U_{N}^{+}$ and the free group $\mathbb{F}_{N}$.

Recall that $\C[\mathbb{F}_{N}]$ is the complex $*$-algebra of finite linear combinations of elements of $\mathbb{F}_{N}$, with product induced from the group law and involution induced by the group inverse. Setting $\Delta(g) = g\otimes g$ for any $g\in \mathbb{F}_{N}$ turns this into a Hopf $^*$-algebra associated with a compact quantum group, with counit $\varepsilon$ equal to the trivial representation, $\varepsilon(g)=1$ for any $g\in \mathbb{F}_{N}$. Moreover, if $g_{1}, \cdots, g_{N}$ are free generators of $\mathbb{F}_{N}$, then the matrix $\mathrm{diag}(g_{1}, \cdots, g_{N})$ satisfies the generating relations of $\Pol(U_{N}^{+})$, hence there is a surjective $*$-homomorphism $q:\Pol(U_{N}^{+})\to \C[\mathbb{F}_{N}]$ sending $u_{ij}$ to $\delta_{ij}g_{i}$ for all $1\leqslant i\leqslant N$.

\begin{corollary}
Let $d\in \N$, let $v_{1}, \cdots, v_{N}\in \C^{d}$ and let $\alpha_{i}\in i\R$ for all $1\leqslant i\leqslant N$. Then, there exists a unique Gaussian pair $(\phi, \eta)$ on $\C[\mathbb{F}_{N}]$ such that if $g_{1}, \cdots, g_{N}$ denote the generators of the free group,
\begin{equation*}
\eta(g_{i}) = v_{i} \quad \& \quad \phi(g_{i}) = \alpha_{i} - \frac{1}{2}\|v_{i}\|^{2}, \;\; i=1,\ldots,N.
\end{equation*} 
Moreover, all Gaussian pairs with surjective cocycle on $\C[\mathbb{F}_{N}]$ arise in that way.
\end{corollary} 

\begin{proof}
Fix an orthonormal basis $\{e_1,\ldots, e_d\} \in \C^d$.
Let, for $1\leqslant i\leqslant d$, $L_{i}\in M_{N}(\C)$ be the diagonal matrix with coefficients given by $\langle e_r, v_i\rangle$, $1\leqslant r \leqslant d$, and let $H\in M_{N}(\C)$ be the diagonal matrix with coefficients $\alpha_{i}$. Then, $H$ is anti-hermitian and because all the matrices are diagonal, condition \eqref{eq-L-cond} is satisfied. Therefore, we have a Gaussian generating functional $\phi:= \Gamma_W + D_H$ on $\Pol(U_{N}^{+})$. Applying Lemma \ref{lem-ideal} with the sets 
$\mathcal{X}=\{u_{ij}, u_{ij}^{*} \mid 1\leqslant i,j \leqslant N\}$ and $\mathcal{Y} = 
\{u_{ij}, u_{ij}^{*} \mid 1\leqslant i,j \leqslant N, i \neq j\}$ shows after an easy computation that this functional vanishes on the kernel of the homomorphism $q$, hence by Proposition \ref{prop:quotient} yields a Gaussian generating functional on 
$\C[\mathbb{F}_{N}]$.

We now have to prove that any Gaussian pair $(\phi, \eta)$ on $\Pol(U_{N}^{+})$ which vanishes on the ideal generated by the set $\mathcal{Y}$ above is of the form in the statement. To do this, write $\phi=\Gamma_W+ D_H$ and first observe that because $\eta(u_{ij}) = 0$ for $i\neq j$, the corresponding matrix $L_{r}$ must be diagonal for all $1\leqslant r\leqslant d$. We now simply set
\begin{equation*}
v_{i} = \sum_{r=1}^{d}L_{ii}e_{r}, \;\;\; 1\leqslant N,
\end{equation*}
to get $\eta(g_{i}) = v_{i}$. Moreover, this implies that for all $1\leqslant i, j\leqslant N$,
\begin{equation*}
\Gamma_{W}(u_{ij}) = -\frac{\delta_{ij}}{2}\|v_{i}\|^{2},
\end{equation*}
so that the condition $\phi(u_{ij}) = 0$ yields $H_{ij} = 0$ for $i\neq j$. In other words, $H$ is also diagonal and because it is anti-hermitian, its entries are pure imaginary numbers. Denoting for $1\leqslant i \leqslant N$  by $\alpha_{i}$ its $i$-th coefficient, we get the second formula in the statement, and the proof is complete.
\end{proof}

\section{Free orthogonal quantum groups}\label{Sec:orthogonal}

We will now classify Gaussian processes on another family of compact quantum groups, namely the free orthogonal ones. However, this time there are two families of free orthogonal quantum groups of Kac type, which have to be dealt with separately. Nevertheless, the general strategy is the same and we therefore first gather some general tools.

Again let us fix $N \in \N$ and assume that $\QG$ is a compact quantum subgroup of $U_{N}^{+}$, so that there is a surjective Hopf $*$-homomorphism $q: \Pol(U_{N}^{+})\to \Pol(\QG)$. By Proposition \ref{prop:quotient}, to characterize Gaussian pairs on $\QG$ we need to describe these Gaussian pairs on $U_{N}^{+}$ which vanish on $\ker(q)$; we will naturally exploit Lemma \ref{lem-ideal}.

\subsection{The standard case}

As already mentioned, there exist two types of free orthogonal quantum groups of Kac type. We start with the simplest and most studied one, that we therefore term ``standard''. The Hopf $*$-algebra $\Pol(O_{N}^{+})$ is the quotient of $\Pol(U_{N}^{+})$ by the relations $u_{ij}^{*} = u_{ij}$ for $1\leqslant i, j\leqslant N$, so that $\Pol(O_{N}^{+})$ is the universal $*$-algebra generated by $N^{2}$ self-adjoint elements $(u_{ij})_{1\leqslant i, j\leqslant N}$ such that the matrix $U = (u_{ij})_{1\leqslant i,j\leqslant N}$ is unitary. To describe Gaussian pairs on $\Pol(O_{N}^{+})$ we will thus use the results of the last section, Proposition \ref{prop:quotient}, and Lemma \ref{lem-ideal}. 

\begin{theorem}\label{thm-gauss-ONplus}
Let $d\in\N$, and  let $L_{1}, \cdots, L_{d}\in M_{N}(\C)$ be anti-symmetric matrices such that
\begin{equation}\label{eq-cond-real}
\sum_{r=1}^{d} \overline{L_{r}}L_{r} \in M_{N}(\R).
\end{equation}
Further, let $H\in M_{N}(\R)$ be anti-symmetric and set
\begin{equation*}
W = \sum_{r=1}^{d} L_{r}\otimes L_{r}^{*} \in M_{N}(\C)\otimes M_{N}(\C).
\end{equation*}
Denote by $(e_{1}, \cdots, e_{d})$ the usual orthonormal basis of $\C^{d}$. Then, there exists a Gaussian cocycle $\eta : \Pol(O_{N}^{+}) \to \C^{d}$ and a hermitian functional $\Gamma_{W} : \Pol(O_{N}^{+}) \to \C$ such that for every $1\leqslant i,j\leqslant N$ we have
\begin{align*}
\eta(u_{ij}) = \sum_{r=1}^{d} (L_{r})_{ij} e_{r};\\
\Gamma_W(u_{ij}) = - \frac{1}{2} \sum_{r=1}^{d} (L_{r}^{*} L_{r})_{ij},
\end{align*}
and for all $a, b\in O_{N}^{+}$,
\begin{equation*}
\partial{\Gamma_W}  (a^* \otimes b)= \langle \eta(a), \eta(b)\rangle.
\end{equation*}
The conditions above determine the pair $(\Gamma_W, \eta)$ uniquely. Set $\phi = \Gamma + D_{H}$. Then $(\phi, \eta)$ is a Gaussian pair on $\Pol(O_{N}^{+})$. Moreover, all Gaussian pairs with surjective cocycles on $\Pol(O_{N}^{+})$ (hence also all Gaussian generating functionals) arise in this way for unique $W$ and $H$.
\end{theorem}

\begin{proof}
Observe first that if matrices $L_{1}, \cdots, L_{d} \in M_{N}(\C)$ are assumed to be antisymmetric, then condition \eqref{eq-cond-real} implies condition \eqref{eq-L-cond}. Indeed, we then have 
\begin{align*}
\sum_{r=1}^{d} L_{r}^{*}L_{r} & = \left(\sum_{r=1}^{d} L_{r}^{*} L_{r}\right)^{*} = \left(-\sum_{r=1}^{d} \overline{L}_{r} L_{r} \right)^{*} = \left(-\sum_{r=1}^{d} \overline{L}_{r} L_{r}\right)^{t} = - \sum_{r=1}^{d}  L_{r}^{t}\overline{L}_{r}^{t} \\
& = -\sum_{r=1}^{d}  L_{r}\overline{L}_{r} = \sum_{r=1}^{d} L_{r} L_{r}^{*}.
\end{align*}
	 
The remarks in the beginning of this subsection together with Theorems \ref{thm-gauss-UNplus} and \ref{th:unitaryW} imply that to prove the theorem it suffices to show that given $d\in \N$ and matrices $L_{1}, \cdots, L_{d}\in M_{N}(\C)$ and $H \in M_{N}(\R)$ satisfying the conditions in Theorem \ref{thm-gauss-UNplus}, the associated generating functional $\tilde{\phi} = \Gamma_{W} + D_{H} : \Pol(U_{N}^{+})\to \C$ factors via the ideal generated by the relations $u_{ij} = u_{ij}^{*}$ for $1\leqslant i,j\leqslant N$ if and only the matrices $L_{1}, \cdots, L_{d}$ and $H$ satisfy the conditions in the statement of the theorem. Note that in the last sentence, and everywhere below in the proof, we view $u_{ij}$ as elements of $\Pol(U_{N}^{+})$. 
	
Let us start by noticing that by conditions \eqref{Gammaad} and \eqref{phitoH}, written in matrix form, we have for $1\leqslant i, j\leqslant N$ that $\tilde{\phi}(u_{ij}) = \tilde{\phi}(u_{ij}^{*})$
if and only if 	
\begin{equation*}\label{eq-1st}
\phi^{(N)}(U) = -\frac{1}{2}M(W) + H = - \frac{1}{2} M(W)^{t} - H^{t} = \phi^{(N)}(\overline{U}). 
\end{equation*}
Since $M(W)$ is hermitian while $H$ is anti-hermitian, taking adjoints in the previous equality yields
\begin{equation*}
\frac{1}{2} M(W) - H = \frac{1}{2} M(W)^{t} + H^{t}
\end{equation*}
and adding the two equations shows that $M(W)$ is symmetric -- hence real-valued -- and similarly $H$ is antisymmetric, hence real-valued. Consider next the sets
\begin{align*}
\mathcal{X} & = \{u_{ij}, u^{*}_{ij} \mid 1\leqslant i, j\leqslant N\}\subset \Pol(U_{N}^{+}), \\
\mathcal{Y} & = \{u_{ij} - u^{*}_{ij} \mid 1\leqslant i, j\leqslant N\}\subset \ker(\varepsilon),  
\end{align*}
to which we will apply Lemma \ref{lem-ideal}. Let $1\leqslant i,j,k,l\leqslant N$ and look at the condition
\begin{equation*}
\tilde\phi(u_{ij}u_{kl}) = \tilde{\phi}(u_{ij} u_{kl}^{*}).
\end{equation*}	
As we already argued that $\phi(u_{ij}) =\phi(u_{ij}^{*})$, the displayed equality is equivalent to 	
\begin{equation*}
\partial \tilde{\phi}(u_{ij} \otimes u_{kl}) = \partial\tilde{\phi}(u_{ij} \otimes u_{kl}^{*}),
\end{equation*}		
hence to
\begin{equation*}
\langle \tilde\eta(u_{ij}^{*}), \tilde\eta(u_{kl})\rangle = \langle \tilde{\eta}(u_{ij}^{*}), \tilde{\eta}(u_{kl}^{*})\rangle,
\end{equation*}
and to
\begin{equation*}
- \langle \tilde{\eta}(u_{ji}), \tilde{\eta}(u_{kl})\rangle = - \langle \tilde{\eta}(u_{ji}), - \tilde{\eta}(u_{lk})\rangle,
\end{equation*}
where $\tilde{\eta} : \Pol(U_{N}^{+})\to \C^{d}$ is the associated surjective cocycle. But then we must have $\tilde{\eta}(u_{lk}) = - \tilde\eta(u_{kl})$, i.e.\ each of the matrices $L_{r}$, $1\leqslant r\leqslant d$ must be antisymmetric. An analogous argument shows that if $L_{1}, \cdots, L_{d}$ are antisymmetric, then we also have
\begin{equation*}
\tilde{\phi}(u_{ij}u_{kl}) = \tilde{\phi}(u_{ij}^{*}u_{kl})
\end{equation*}
for all $1\leqslant i,j,k,l\leqslant N$. As $\tilde{\phi}$ is hermitian, the other conditions required in Lemma \ref{lem-ideal} also hold, and the proof is complete.
\end{proof}

\begin{remark}
As in the case of $U_{N}^{+}$, we see that drifts are given by antisymmetric real matrices, which form the Lie algebra $\mathfrak{o}_{N}$ of the classical compact Lie group $O_{N}$.
\end{remark}

Before turning to the next case, let us comment on the consequences of Theorem \ref{thm-gauss-ONplus} for bistochastic quantum groups. These are free easy quantum groups denoted by $B_{N}^{+}$, $B_{N}^{+\prime}$ and $B_{N}^{+\sharp}$. We refer to \cite{Web12} for the definition of these objects and the proof that they are isomorphic to $O_{N-1}^{+}$, $O_{N-1}^{+}\times \Z_{2}$ and $O_{N-1}^{+}\ast \Z_{2}$ respectively. Using these isomorphisms, their Gaussian pairs are easily described. 

\begin{proposition}
The Gaussian pairs on $B_{N}^{+}$, $B_{N}^{+\prime}$ and $B_{N}^{+\sharp}$ are in a natural one-to-one correspondence with Gaussian pairs on $O_{N-1}^{+}$.
\end{proposition}

\begin{proof}
The result is trivial for $B_{N}^{+}$. For the other two quantum groups, $\QG= B_N^{+\prime}$ or $\QG=B_{N}^{+\sharp}$, denote by $\gamma$ a symmetry (i.e. a self-adjoint unitary) generating $\Z_{2}$ and by $p = (1+\gamma)/2$ the corresponding projection in $\Pol(\QG)$. By \cite[Lem 4.4]{FFS23}, the Hopf $*$-ideal
\begin{equation*}
K_{\infty} = \bigcap_{n\in \N}K_{n}
\end{equation*}
contains $p$ (observe that $\varepsilon(p) = 0$). As a consequence, any Gaussian pair factors through the quotient by the Hopf $*$-ideal generated by $p$, and that quotient is nothing but $\Pol(B_{N}^{+})$. Hence the result.
\end{proof}

\subsection{The symplectic case}

Consider the matrix
\begin{equation*}
J_{N} = \left(\begin{array}{cc}
0 & I_{N}  \\
-I_{N} & 0
\end{array}\right) \in M_{2N}(\C).
\end{equation*}
The Hopf $*$-algebra of the compact quantum group $O_{J_{N}}^{+}$ is defined by taking the quotient of $\Pol(U_{2N}^{+})$ by the $*$-ideal generated by the relations $U = J\overline{U}J^{-1}$. In block matrix form, with
\begin{equation*}
U = \left(\begin{array}{cc}
V & X \\
Y & Z
\end{array}\right),
\end{equation*}
this can be written as
\begin{equation*}
\left(\begin{array}{cc}
V & X  \\
Y & Z
\end{array}\right)
=
\left(\begin{array}{cc}
\overline{Z} & -\overline{Y} \\
-\overline{X} & \overline{V}
\end{array}\right).
\end{equation*}
(note that $J^{-1}=-J$). Even though $O_{J_{N}}^{+}$ belongs to the general family of free orthogonal quantum groups, it is very natural to call it the \emph{free symplectic quantum group}. We will again use Lemma \ref{lem-ideal} to determine the Gaussian processes on $O_{J_{N}}^{+}$, using this time the sets
\begin{align*}
\mathcal{X} & = \{u_{ij}, u^{*}_{ij} \mid 1\leqslant i, j\leqslant 2N\}\subset \Pol(U_{2N}^{+}); \\
\mathcal{Y} & = \{u^{*}_{ij} - u_{i+N, j+N}, u^{*}_{i+N, j} + u_{i,j+N} \mid 1\leqslant i, j \leqslant N\} \subset \ker(\varepsilon). 
\end{align*}

\begin{theorem}\label{thm-gauss-OJplus}	
Let $d\in\N$, let $L_{1}, \cdots, L_{d}\in M_{2N}(\C)$ be matrices such that $L_{r}^{t} = J_{N}L_{r}J_{N}$ for each $1\leqslant r\leqslant d$ and the matrix
\begin{equation*}
W = \sum_{r=1}^{d} L_{r}\otimes L_{r}^{*} \in M_{2N}(\C)\otimes M_{2N}(\C)
\end{equation*}
satisfies $JM(W)J = -M(W)^{t}$. Let moreover $H\in M_{N}(\C)$ be an anti-hermitian matrix such that $JHJ = H^{t}$. Denote by $(e_{1}, \cdots, e_{d})$ the usual orthonormal basis of $\C^{d}$. Then, there exists a Gaussian cocycle $\eta : \Pol(O_{J_{N}}^{+}) \to \C^{d}$ and a hermitian functional $\Gamma_{W} : \Pol(O_{J_{N}}^{+}) \to \C$ such that for every $1\leqslant i,j\leqslant N$
\begin{align*}
\eta(u_{ij}) = \sum_{r=1}^{d} (L_{r})_{ij} e_{r};\\
\Gamma_{W}(u_{ij}) = -\frac{1}{2} \sum_{r=1}^{d} (L_{r}^{*} L_{r})_{ij}
\end{align*}
and for all $a, b \in \Pol(O_{J_{N}}^{+})$,
\begin{equation*}
\partial{\Gamma_{W}} (a^{*}\otimes b) = \langle \eta(a), \eta(b) \rangle.
\end{equation*}
The conditions above determine the pair $(\Gamma_W, \eta)$ uniquely. Set $\phi = \Gamma_W + D_{H}$. Then $(\phi, \eta)$ is a Gaussian pair on $\Pol(O_{J_{N}}^+)$. Moreover all Gaussian pairs with surjective cocycle on $\Pol(O_{J_{N}}^+)$ (hence also all Gaussian generating functionals) arise in this way, for unique $W$ and $H$.
\end{theorem}

\begin{proof}
As the logic of the proof is identical to that of Theorem \ref{thm-gauss-ONplus} we will just sketch the arguments.	We need to view first $\tilde{\phi} = \Gamma_{W} + D_{H}$ as a Gaussian generating functional on $\Pol(U_{2N}^+)$. To do this, observe that
\begin{align*}
\sum_{r=1}^{d}L_{r}^{*}L_{r} & = M(W)= -JM(W)^{t}J = -\sum_{r=1}^{d} JL_{r}^{t}\overline{L_{r}}J = \sum_{r=1}^{d}JL_{r}^{t}J^{2}\overline{L_{r}}J  \\
& = \sum_{r=1}^{d}L_{r}L_{r}^{*}
\end{align*}
so that condition \eqref{eq-L-cond} holds.

We now have to find conditions ensuring that $\tilde{\phi}$ vanishes on suitable elements and for that purpose it is convenient to use the matrix form of some equalities. Begin by noting that if $\tilde{\phi}$ vanishes on the ideal of interest, we must have
\begin{equation*}
\tilde{\phi}^{(2N)}(U)  = \tilde{\phi}^{(2N)}(J\overline{U}J^{-1}) = - J\phi^{(2N)}(\overline{U})J.
\end{equation*}
Using again the arguments as before \eqref{eq-1st} we see that this is equivalent to
\begin{equation*}
\frac{1}{2} M(W) + H = -J\left(\frac{1}{2}M(W)^{t} - H^{t}\right)J
\end{equation*}
Taking the adjoint and using the fact that $M(W)^{*} = M(W)$, $H^{*} = -H$, and $J^{*} = -J$, we get
\begin{equation*}
\frac{1}{2} M(W) - H = -J\left(\frac{1}{2}\overline{M(W)} - \overline{H}\right)J = -J\left(M(W)^{t} + H^{t}\right)J.
\end{equation*}
Subtracting and adding the last two equalities we finally deduce that
\begin{eqnarray}\label{eq-cond-J}
M(W) & = & -JM(W)^{t}J,  \nonumber \\
H & = & JH^{t}J.
\end{eqnarray}
On the other hand one can check that the other conditions needed to apply Lemma \ref{lem-ideal} amount to the equality
\begin{equation}\label{eq-cond-W}
W = (I\otimes J) W^{\mathrm{id}\otimes t} (I\otimes J) = (J\otimes I) W^{t\otimes\mathrm{id}} (J\otimes I).
\end{equation}
For
\begin{equation*}
W = \sum_{r=1}^{d} L_{r}\otimes L_{r}^{*},
\end{equation*}
condition \eqref{eq-cond-W} implies $L_{r} = JL_{r}^{t}J$ for each $1\leqslant r\leqslant d$ (since the $L_{r}$'s are linearly independent) and this completes the proof.
\end{proof}

\begin{remark} 
The conditions on the matrix $H$ listed in the statement mean exactly  that it belongs to the Lie algebra $\mathfrak{sp}(N)$ of the compact symplectic group $Sp(N) = Sp(2N,\mathbb{C})\cap U(2N)$. Hence, once again, drifts are given by the Lie algebra of the corresponding classical group.
\end{remark}

\begin{remark}
As in the last section, there is a more general family of free orthogonal quantum groups which are not of Kac type. The maximal Kac type quantum subgroup is then a free product of copies of $O_{N}^{+}$, $O_{J_{N}}^{+}$ and $U_{N}^{+}$ (see \cite{dfs-rfd}), so that once again the missing information concerns Gaussian processes on free products.
\end{remark}

\section{Applications to central functionals}
\label{sec:central}

Following the non-commutative philosophy, states on the $*$-algebra $\Pol(\QG)$ of a compact quantum group can be thought of as probability measures on $\QG$. In this picture, convolution semigroups are analogues of continuous processes. In the case of a compact Lie group $G$, the Brownian motion yields a specific continuous process whose probabilistic properties are closely linked to the structure of $G$.

For the quantum groups concepts used below we refer for example to \cite{freslon23}.
Recall that a functional $\phi : \Pol(\QG)\to \C$ is \emph{central} if for any other functional $\psi:\Pol(\QG) \to \C$ we have $\phi\ast \psi = \psi\ast\phi$. This is equivalent to the following (see for instance \cite[Prop 6.2 and Prop 6.9]{cfk}): for any irreducible representation $\alpha$ of $\QG$ with representative $u^{\alpha}\in M_{\dim(\alpha)}(\Pol(\QG))$, there exists $c_{\alpha}\in \C$ such that for all $1\leqslant i, j\leqslant \dim(\alpha)$,
\begin{equation*}
\varphi(u^{\alpha}_{ij}) = \delta_{ij}\frac{c_{\alpha}}{\dim(\alpha)}.
\end{equation*}
As a consequence, a central functional is completely determined by its values on irreducible characters: $\varphi(\chi_{\alpha}) = c_{\alpha}$, $\alpha \in \textup{Irr}(\QG)$.

If $G$ is a classical compact group, then the state coming from integration with respect to a probability measure $\mu$ on $G$ is central if and only if $\mu$ is conjugation invariant. This is in particular the case for the Brownian motion, and conversely, Liao classified continuous processes of conjugation invariant measures in terms of Brownian motion in \cite{liao}.

The following easy observation shows how one can build central Gaussian generating functionals.

\begin{lemma}
	\label{lem:cocentral}
Suppose that $\QG$ is a compact quantum group and $\QH$ is a co-central quantum subgroup, i.e.\ that we have a surjective Hopf*-homomorphism $q:\Pol(\QG)\to \Pol(\QH)$ such that 
\[(q\otimes \mathrm{id})\circ\Delta_{\QG} = \sigma\circ(\mathrm{id}\otimes q)\circ \Delta_{\QG},\]
where $\sigma$ denotes the tensor flip.
Then any Gaussian generating functional $\phi:\Pol(\QH) \to \C$ (necessarily central) yields a central Gaussian generating functional $\widetilde{\phi} = \phi \circ q$ on $\QG$. 	
\end{lemma}
\begin{proof}
This is an elementary computation: for any linear functional $\psi$, we have
\begin{align*}
\widetilde{\phi}\ast \psi & = (\widetilde{\phi}\otimes \psi)\circ\Delta = (\phi\otimes \psi)\circ(q\otimes \mathrm{id})\circ\Delta \\
& = (\phi\otimes \psi)\circ\sigma\circ(\mathrm{id}\otimes q)\circ\Delta = (\psi\otimes \phi)\circ(\mathrm{id}\otimes q)\circ\Delta \\
& = (\psi\otimes \widetilde{\phi})\circ\Delta = \psi\ast\widetilde{\phi}.
\end{align*}
\end{proof}

Note that cocentral quantum subgroups are neccessarily abelian (i.e.\ their coproduct is invariant under the flip). It is easy to check that if we denote the canonical generator of $\Pol(\mathbb{T})$ by $z$, then the map $u_{ij}\mapsto \delta_{ij} z$, $i,j=1,\ldots,n$, makes $\mathbb{T}$ a co-central quantum subgroup of $U_N^+$ for any $N \in \mathbb{N}$.

\subsection{Central Gaussian processes}

Our goal in this final section is to explore the relationship between Gaussianity and centrality for generating functionals on free easy quantum groups. Our first result is rather negative: there are no central Gaussian functionals on these compact quantum groups, except for the `trivial' ones which come from classical central subgroups. The proof relies on our classification results but the details differ depending on the quantum groups involved. For $O_{N}^{+}$, this means that there is no central Gaussian process, and this was first proven in \cite[Theorem 3.23]{dg12}. However, our results enable us to recover this in a very simple way, and the proof applies also for $O_{J_{N}}^{+}$.

We fix throughout the section $N \in \N$, $N\ge 2$.

\begin{proposition}
There is no non-zero central Gaussian generating functional on $O_{N}^{+}$ and $O_{J_{N}}^{+}$ for any $N\geqslant 2$.
\end{proposition}

\begin{proof}
We start with $O_{N}^{+}$ and consider, according to Theorem \ref{thm-gauss-ONplus}, a central Gaussian generating functional $\phi = \Gamma_{W} + D_{H}$. Let us furthermore denote by $U=(u_{ij})_{i,j=1}^N$ the fundamental representation of $O_{N}^{+}$. Recall the discussion before Definition \ref{def:drift};
for simplicity we will simply write $\phi(U)$ and $\phi(U\tilde{\otimes}U)$ for scalar matrices denoted there $\phi^{(N)}(U)$ and $\phi^{(N^2)}(U \tilde{\otimes} U)$. 

As $U$ is irreducible, we have
\begin{equation*}
-\frac{1}{2}M(W) + H = \phi(U) = \lambda I_{N}
\end{equation*}
for some $\lambda \in \mathbb{R}$.
Taking transposes and remembering that $M(W)$ is symmetric while $H$ is anti-symmetric yields
\begin{equation*}
-\frac{1}{2}M(W) - H = \left(-\frac{1}{2}M(W) + H\right)^{t} = (\lambda I_{N})^{t} = \lambda I_{N} = -\frac{1}{2}M(W) + H,
\end{equation*}
so that $H = 0$ and $M(W) = -2\lambda I_{N}$.

Let us now consider the vector space
\begin{equation*}
\mathcal{V} := \mathrm{Span}\{u_{ij}u_{kl} \mid 1\leqslant i, j, k, l \leqslant N\}.
\end{equation*}
This is the space of coefficients of the representation $U\otimes U$ of $\QG$ and since the latter representation the sum of two irreducible ones, the space of restrictions to $\mathcal{V}$ of central linear functional is at most two-dimensional. We now claim that the restrictions of the counit $\varepsilon$ and the Haar state $h$ to $\mathcal{V}$ are linearly independent. Indeed,
\begin{equation*}
0 = \varepsilon(u_{12}u_{12}) \neq h(u_{12}u_{12}) = \frac{1}{N},
\end{equation*}
proving our claim. This implies that there exist $\alpha, \beta\in \C$ such that $\phi = \alpha\varepsilon_{\QG} + \beta h_{\QG}$. Let us use this to compute $\phi(U\tilde{\otimes}U)$ using the formula given before Proposition \ref{prop:classical}:
\begin{align*}
-\frac{1}{2}M(W)\otimes I_{N} - \frac{1}{2}I_{N}\otimes M(W) - W^{t\otimes t} & = \phi(U)\otimes I_{N} + I_{N}\otimes \phi(U) - W^{t\otimes t}  = \phi(U\tilde{\otimes}U) \\
& = \alpha\varepsilon(U\tilde{\otimes}U) + \beta h(U\tilde{\otimes}U) \\
& = \alpha I_{N}\otimes I_{N} + \beta\sum_{i=1}^{N}E_{ii}\otimes E_{ii}.
\end{align*}
To conclude, we will use two elementary facts. First, because $\phi(1) = 0$, we have $\alpha + \beta = 0$. Second, applying the linear map $M$ to the equality above yields (remember that $L_{r}$ is antisymmetric, so that $M(W^{t\otimes t}) = M(W)$)
\begin{equation*}
-2M(W) = (\alpha + \beta)I_{N} = 0.
\end{equation*}
Hence in the end $M(W) = 0$, which readily implies that $W=0$ and eventually $\phi=0$.

For $O_{J_{N}}^{+}$ the proof is similar and we only sketch it. Once again we have
\begin{equation*}
-\frac{1}{2}M(W) + H = \phi(U) = \lambda I_{2N}
\end{equation*}
for some $\lambda \in \mathbb{C}$,
and conjugating the transpose of the last equality by $J_{N}$ yields
\begin{equation*}
\lambda I_{N} = -J\left(-\frac{1}{2}M(W)^{t} + H^{t}\right)J = -\frac{1}{2}M(W) - H,
\end{equation*}
so that $H = 0$ and $M(W) = -2\lambda I_{2N}$. Decomposing the restriction of $\phi$ to the space $\mathcal{V}$ of coefficients of $U\otimes U$, we get
\begin{equation*}
M(W^{t\otimes t}) = -M(W).
\end{equation*}
Since $JL_{r}J = L_{r}^{t}$, we have
\begin{equation*}
M(W^{t\otimes t}) = \sum_{r=1}^{d}L_{r}^{t}L_{r}^{\ast t} = \sum_{r=1}^{d}JL_{r}JJL_{r}^{\ast}J = -JM(W)J = M(W)^{t}
\end{equation*}
and since we already know that $M(W) = 2\lambda I_{N}$, we conclude that $M(W) = -M(W)$. Hence $M(W) = 0$ and we finish the proof as in the first part.
\end{proof}

In the case of $U_{N}^{+}$ and $H_{N}^{\infty +}$, there are central Gaussian functionals, but they all come from classical subgroups. To prove this, we need a small linear algebra result.

\begin{lemma}\label{lem:sumcentral}
Let $L_{1}, \cdots, L_{d}, Z\in M_{N}(\C)$ be such that
\begin{equation*}
\sum_{r=1}^{d}L_{r}^{*}\otimes L_{r}\in \C.Z^{*}\otimes Z.
\end{equation*}
Then, $L_{r}\in \C.Z$ for all $1\leqslant r\leqslant d$.
\end{lemma}

\begin{proof}
Let us consider a maximal linearly independent family of matrices in $\{L_{1}, \cdots, L_{d}\}$ which is linearly independent from $Z$. If it is empty, the result follows. Otherwise, up to renumbering we may assume (allowing $k=d$) that the matrices $L_{1}, \cdots, L_{k}, Z$ are linearly independent and that for all $k+1\leqslant j\leqslant d$, $L_{j}\in \mathrm{Lin}\{L_{1}, \cdots, L_{k}, Z\}$ so that there exist $\alpha_{j}\in \C$ and $Z_{j}\in\mathrm{Lin}\{L_{2}, \cdots, L_{k}, Z\}$ such that
\begin{equation*}
L_{j} = \alpha_{j}L_{1} + Z_{j}.
\end{equation*}
Let $f : M_{N}(\C)\to \C$ be a linear functional such that $f(L_{1}) = 1$ and $f(L_{i}) = 0 = f(Z)$ for all $1\leqslant i\leqslant k$. Then, we have
\begin{align*}
0 & = (\mathrm{id}\otimes f)(Z^{*}\otimes Z) = \sum_{r=1}^{d}L_{r}^{*}f(L_{r}) = L_1^* + \sum_{r=k+1}^{d}\alpha_{r}L_{r}^{*} = L_1^* + \sum_{r=k+1}^{d}\left(\vert\alpha_{r}\vert^{2}L_{1}^{*} + \alpha_{r}Z_{r}^{*}\right). \\
\end{align*}
Taking adjoints yields a vanishing linear combination where the coefficient of $L_{1}$ is
\begin{equation*}
1 + \sum_{r=2}^{k}\vert\alpha_{r}\vert^{2} > 0,
\end{equation*}
contradicting linear independence of $\{L_1,\ldots,L_k, Z\}$.
\end{proof}

We also need a small quantum group result.

\begin{lemma}\label{lem:diagonal}
Let $\phi=\Gamma_W+ D_H $ be a Gaussian generating functional on $U_{N}^{+}$ such that $W$ and $H$ are both multiples of the identity matrices. Then, $\phi$ factors through $\mathbb{T}$, the subgroup of $U_{N}$ consisting in scalar matrices.
\end{lemma}

\begin{proof}
First notice that since $W = \sigma W$, $\phi$ factors trough the abelianization of $U_{N}^{+}$ by Proposition \ref{prop:classical}. The key observation now is that under the assumption in the statement, both $\phi$ and $\eta$ vanish on all elements of the form $u_{ii} - u_{jj}$ or $u_{ij}$ for $1\leqslant i\neq j\leqslant N$. As a consequence, applying Lemma \ref{lem-ideal} with
\begin{equation*}
\mathcal{X} = \{u_{ij} \mid 1\leqslant i, j\leqslant N\} \quad \& \quad \mathcal{Y} = \{u_{ii} - u_{jj}, u_{ii}^{*} - u_{jj}^{*}, u_{ij}, u_{ij}^{*} \mid 1\leqslant i\neq j\leqslant N\}
\end{equation*}
yields that $\phi$ factors through the quotient by the ideal generated by $\mathcal{Y}$, which is nothing but the subgroup of scalar matrices in the abelianization of $U_{N}^{+}$.
\end{proof}

We are now ready for the last two cases.

\begin{proposition}
The only central Gaussian functionals on $U_{N}^{+}$ and $H_{N}^{\infty +}$ are those coming from $\mathbb{T}$, viewed as a quantum subgroup. In particular they are determined by the two parameters, $\nu \in \mathbb{R}$ and $\mu \geq 0$, which correspond to the drift and the variance parameter of the underlying (classical) Brownian on $\mathbb{T}$.
\end{proposition}

\begin{proof}
We start with $U_{N}^{+}$ and let $\phi = \Gamma_{W} + D_H$ be central. As above, evaluating on $U$ yields $\lambda\in \C$ such that $M(W) + H = \lambda I_{N}$ and taking adjoints shows that $H = \mathrm{Im}(\lambda)I_{N}$ and $M(W) = \mathrm{Re}(\lambda)I_{N}$. As $H$ is antihermitian, we have $H= i \nu$, $\nu \in \mathbb{R}$. Moreover, evaluating on $U\otimes U$ must give an element of the center of $M_{N}(\C)\otimes M_{N}(\C)$. Hence there is $\mu\in\C$ such that
\begin{equation*}
W = \sum_{r=1}^{d}L_{r}\otimes L_{r}^{*} = \mu I_{N}\otimes I_{N}.
\end{equation*}
It is easy to see that $\mu \geq 0$.
By Lemma \ref{lem:sumcentral}, each matrix $L_{r}$ is a multiple of the identity and by linear independence we have $d = 1$ and $L_{1} = \sqrt{\mu}I_{N}$. In particular, $W = \Sigma W$ so that the functional factors through $U_{N}$ by Proposition \ref{prop:classical}. Applying then Lemma \ref{lem:diagonal}, we conclude that if it factors through the subgroup of scalar matrices in $U_{N}$, which is isomorphic to $\mathbb{T}$. 

Conversely, if $\phi$ is any Gaussian generating functional on $\mathbb{T}$, then $\phi\circ\pi$ is a central Gaussian generating functional on $U_{N}^{+}$ by Lemma \ref{lem:cocentral} and remarks after that.

Let us now consider $H_{N}^{\infty +}$. Since both its fundamental representation $U$ and $U\otimes U$ are irreducible, the same argument as for $U_{N}^{+}$ shows that $W$ and $H$ must be multiples of the identity. Therefore, $\phi$ factors through the group of scalar matrices in the classical version of $H_{N}^{\infty}$, which is once again $\mathbb{T}$.
\end{proof}

\subsection{Centralizing Gaussian processes}

Even though Gaussian processes are seldom central on free easy quantum groups, one can make them central in the following way. Assume that $\QG$ is of Kac type and let $\mathbb{E} : \Pol(\QG) \to \Pol(\QG)_{0}$ be the conditional expectation onto the $*$-subalgebra of characters (see for example \cite[Section 2]{FSW}). Then, for any generating functional $\phi$,
\begin{equation*}
\widetilde{\phi} = \phi\circ\mathbb{E}
\end{equation*}
is a central generating functional. Of course, the original Gaussianity is in general lost in the process, but still entails specific constraints on $\widetilde{\phi}$ and as we will now see, the class of generating functionals obtained in that way has remarkable properties.

In the case of free orthogonal quantum groups the computations are simpler. Indeed, $\Pol(O_{N}^{+})_{0}$ and $\Pol(O_{J_{N}}^{+})_{0}$ are generated by the elements $\chi_{U^{\otimes p}}$, $p\in \N$, and by Gaussianity the value of $\phi$ in these characters is determined by the values for $p = 1$ and $p = 2$.

We begin with a very general observation.

\begin{lemma}\label{lem-moments-phi}
Let $\QG$ be a compact quantum group and let $\phi$ be a Gaussian generating functional on $\Pol(\QG)$, and let $U = (u_{ij})_{1\leqslant i, j\leqslant N}$ be a representation of $\QG$. Setting
\begin{equation*}
\phi_{1} = \phi(\chi_{U}) = \sum_{j=1}^{N} \phi(u_{ii}) \quad \& \quad \phi_{2} = \phi(\chi_{U\otimes U}) = \sum_{i, j=1}^{N} \phi(u_{ii} u_{jj}),
\end{equation*}
we have
\begin{equation*}
\phi(\chi_{U^{\otimes p}}) = \frac{p(p-1)}{2} N^{p-2} \phi_{2} - p(p-2) N^{p-1} \phi_{1}
\end{equation*}
for all $p\in\N$.
\end{lemma}

\begin{proof}
Using \cite[Prop 2.7]{FFS23}, we have
\begin{align*}
\phi(\chi_{U^{\otimes p}}) & = \sum_{j_{1}, \cdots, j_{p} = 1}^{N} \phi(u_{j_{1}j_{1}}\cdots u_{j_{p}j_{p}}) \\
& = \sum_{j_{1}, \cdots, j_{p} = 1}^{N} \sum_{1\leqslant k < l \leqslant p} \phi(u_{j_{k}j_{k}}u_{j_{l}j_{l}})
- \sum_{j_{1}, \cdots, j_{p} = 1}^{N} \sum_{1\leqslant k\leqslant p}\phi(u_{j_{k}j_{k}}) \\
& = \frac{p(p-1)}{2} N^{p-2} \phi_{2} - p(p-2) N^{p-1} \phi_{1}.
\end{align*}
\end{proof}

Applying Lemma \ref{lem-moments-phi} to $O_{N}^{+}$, we see that if $\phi = \Gamma_{W} + D_{H}$, then we have
\begin{eqnarray*}
\phi_{1} & = & \phi(\chi_{U}) = \sum_{i=1}^{N} \phi(u_{ii}) = \mathrm{Tr}\left(\frac{1}{2}M(W)\right) + \mathrm{Tr}(H) = \frac{1}{2} \mathrm{Tr}\left(M(W)\right), \\
\phi_{2} & = & \phi(\chi_{U\otimes U}) = \sum_{j,k=1}^N \phi(u_{jj} u_{kk}) \\
& = & \mathrm{Tr}\otimes\mathrm{Tr}\left(\frac{1}{2} I\otimes M(W) + \frac{1}{2} M(W)\otimes I + W 
+ H\otimes I + I\otimes H \right) \\
& = & N \mathrm{Tr}\left(M(W)\right),
\end{eqnarray*}
because, by anti-symmetry, $\mathrm{Tr}(H) = 0$ and $(\mathrm{Tr}\otimes\mathrm{Tr})(W) = 0$. Therefore, for any $p\in \N$,
\begin{equation*}
\phi(\chi_{U^{\otimes p}}) = \frac{p}{2} N^{p-1} {\rm Tr}\big(M(W)\big).
\end{equation*}

Note that this can also be written as
\begin{equation*}
\phi(\chi_{U^{\otimes p}}) = \frac{{\rm Tr}\big(M(W)\big)}{2} \left.\frac{{\rm d}}{{\rm d}x}\right|_{x=N} x^{p} = \frac{{\rm Tr}\big(M(W)\big)}{2}\phi_{B}(\chi_{U^{\otimes p}}),
\end{equation*}
where $\phi_B$ is a specific generating functional arising from any Gaussian $\phi_{W,H}$ with ${\rm Tr}\big(M(W)\big)=2$. In other words, the centralization of any Gaussian functional (which is not a drift, i.e., with $W\neq 0$) is a positive multiple of a  $\phi_{B}$. It turns out that $\phi_{B}$ is a very peculiar process in at least two other respects.
\begin{itemize}
\item First, it was shown in \cite{BGJ} that if $L : \Pol(SO_{N})\to \C$ is the infinitesimal generator of the usual Brownian motion, i.e. $L(f) = \Delta(f)(I_{N})$, where now $\Delta$ denotes the Laplacian on $SO_N$, then $\phi_{B}$ is proportional to $L\circ q\circ\mathbb{E}$, where $q:\Pol(O_N^+)\to \Pol(SO_N)$ identifies $SO_N$ as a quantum subgroup of $O_N^+$.
\item Second, Liao proved in \cite{liao} that any central generating functional on a compact semisimple Lie group decomposes as the sum of a multiple of the generator of the Brownian motion and a totally discontinuous part. In \cite[Theorem 10.2]{cfk}, the authors proved an analogue of this result: any central generating functional on $O_{N}^{+}$ decomposes as the sum of a multiple of $\phi_{B}$ and a ``jump process''.
\end{itemize}

Turning now to $O_{J_{N}}^{+}$, we have again $\mathrm{Tr}(H) = 0$, $\mathrm{Tr}(L_{r}) = 0$ and $(\mathrm{Tr}\otimes\mathrm{Tr})(W) = 0$. Furthermore,
\begin{equation*}
\phi_{1} = \phi(\chi_{U}) = \frac{1}{2}\mathrm{Tr}\left(M(W)\right) \quad \& \quad \phi_{2} = \phi(\chi_{U\otimes U}) = 2N\mathrm{Tr}\left(M(W)\right)
\end{equation*}
so that once again, for any $p\in \N$,
\begin{equation*}
\phi(\chi_{U^{\otimes p}}) = \frac{p}{2} (2N)^{p-1} \mathrm{Tr}\left(M(W)\right)
\end{equation*}
for any Gaussian generating functional $\phi = \Gamma_{W} + D_{H}$ on $\Pol(O_{J_{N}}^{+})$.

We will conclude with a brief discussion of the case of $U_{N}^{+}$. This time, there are more values to consider to completely determine the generating functional, since we have to evaluate it on all possible tensor products of $U$ and $\overline{U}$. This is easily done and yields
\begin{align*}
\phi(\chi_{U}) & = \frac{1}{2}\mathrm{Tr}\big(M(W)) + \mathrm{Tr}(H), \\
\phi(\chi_{\overline{U}}) & = \frac{1}{2}\rm{Tr}\left(M(W)\right) - \mathrm{Tr}(H), \\
\phi(\chi_{U\otimes U}) & = N\mathrm{Tr}\left(M(W)\right)  - \sum_{r=1}^{d}\left\vert\mathrm{Tr}(L_{r})\right\vert^{2} + 2N\mathrm{Tr}(H), \\
\phi(\chi_{U\otimes \overline{U}}) & = N\mathrm{Tr}\left(M(W)\right) + \sum_{r=1}^{d} \left\vert\mathrm{Tr}(L_{r})\right\vert^{2}, \\
\phi(\chi_{\overline{U}\otimes U}) & = N\mathrm{Tr}\left(M(W)\right) + \sum_{r=1}^{d} \left\vert\mathrm{Tr}(L_{r})\right\vert^{2}, \\
\phi(\chi_{\overline{U}\otimes \overline{U}}) & = N \mathrm{Tr}\left(M(W)\right) - \sum_{r=1}^{d} \left\vert\mathrm{Tr}(L_{r})\right\vert^{2} - 2N\mathrm{Tr}(H).
\end{align*}

Note that these moments depend only on the three parameters
\begin{eqnarray*}
\mathrm{Tr}(H) & \in & i\mathbb{R}, \\
\mathrm{Tr}\left(M(W)\right) & = & \sum_{r=1}^{d} \mathrm{Tr}(L_{r}^{*}L_{r}) \in \R_{+}, \\
(\mathrm{Tr}\otimes \mathrm{Tr})(W) & = & \sum_{r=1}^{d} \left\vert\mathrm{Tr}(L_{r})\right\vert^{2} \in \R_{+}.
\end{eqnarray*}

There is no classification available for central generating functionals on $U_{N}^{+}$, contrary to the case of $O_{N}^{+}$, and therefore no analogue of Liao's result \cite{liao}. Nevertheless, it was shown in \cite{Delhaye} that if $L$ is the infinitesimal generator of the Brownian motion on $U_{N}$, then the corresponding central generating functional on $U_{N}^{+}$, $L\circ q \circ \mathbb{E}$, with $q:\Pol(U_N^+)\to \Pol(U_N)$, is of the form above.


\end{document}